\documentclass[draft]{jotart}
\usepackage{amsmath}
\usepackage{amssymb}
\newcommand{\T}{\mathbb{T}}
\newcommand{\C}{\mathbb{C}}
\newcommand{\D}{\mathbb{D}}
\newcommand{\HP}{\mathcal{P}}
\newcommand{\ran}{\mathrm{ran \ }}
\newcommand{\HM}{\mathcal{M}}
\newcommand{\om}{\omega}
\newcommand{\La}{\langle}
\newcommand{\Ra}{\rangle}

\theoremstyle{proclaim}
\newtheorem{theorem}{Theorem}[section]
\newtheorem{lemma}[theorem]{Lemma}
\newtheorem{corollary}[theorem]{Corollary}

\theoremstyle{fancyproclaim}

\theoremstyle{statement}

\newtheorem{definition}[theorem]{Definition}

\theoremstyle{fancystatement}

\numberwithin{equation}{section}

\providecommand{\AMS}{$\mathcal{A}$\kern-.1667em%
\lower.25em\hbox{$\mathcal{M}$}\kern-.125em$\mathcal{S}$}
\begin{document}% Do not forget this command!
%\issueinfo{vv}{n}{yyyy}
% \issueinfo{vv}{n}{yyyy}, vv is the volume, n is the number, yyyy is the year
% leave these information to fixed later by the editorial office
%\commby{Editor}% Editor is the name of the editor who accepted the article
%\pagespan{101}{103}
% \pagespan{bbb}{eee} where bbb is the beginning page, eee is the ending page
\date{Month dd, yyyy}% This is the date of submission of the article; it
% will be fixed by the editorial office
\revision{Month dd, yyyy}% This is date(s) of revision of the manuscript; it
% will be fixed by the editorial office
\title[Higher order isometric shift operator]{Higher order isometric shift operator on the de Branges-Rovnyak space}% The title of the article
%\dedicatory{Dedicated to ...}% Only if needed
\author[Caixing Gu, Shuaibing Luo]{Caixing Gu, Shuaibing Luo}% For multiple authors please use {\protec \and} sequence
\address{Caixing Gu, Department of Mathematics, California Polytechnic State University, San Luis Obispo, CA 93407, USA}
\email{cgu@calpoly.edu}
\address{Shuaibing Luo, School of Mathematics, and Hunan Provincial Key Laboratory of Intelligent information processing and Applied Mathematics, Hunan University, Changsha, 410082, PR China}
\email{sluo@hnu.edu.cn}
%\address{AUTHOR3, Department, University, Town, Zip, Country}
\begin{abstract} The de Branges-Rovnyak space $H(b)$ is generated by a bounded analytic function $b$ in the unit ball of $H^\infty$. When $b$ is a nonextreme point, the space $H(b)$ is invariant by the forward shift operator $M_z$. We show that the $H(b)$ spaces provide model spaces for expansive quasi-analytic $2n$-isometric operators $T$ with $T^*T - I$ being rank one. Then we describe the invariant subspaces of the $2n$-isometric forward shift operator $M_z$ on $H(b)$.
\end{abstract}
\begin{subjclass}
46E22, 47A15, 47A45. % or "Primary code0, Secondary code1, ...,coden."
\end{subjclass}
\begin{keywords}
de Branges-Rovnyak space; $m$-isometry; operator model; shift invariant subspace.
\end{keywords}
\maketitle

\section*{INTRODUCTION}

Let $\mathbb{D}$ be the open unit disc in the complex plane $\mathbb{C}$, and $\mathbb{T}$ the unit circle. Let $H^{2}$ be the Hardy space on $\mathbb{D}$, and $H^{\infty}$ the space of bounded analytic functions on $\mathbb{D}$ with norm
$$\|f\|_{\infty} = \sup_{z\in \mathbb{D}} |f(z)|.$$
If $P$ is the orthogonal projection from $L^{2}(\mathbb{T})$ onto $H^{2}$, then for $\phi \in L^\infty(\mathbb{T})$, the Toeplitz operator $T_{\phi}: H^{2} \rightarrow H^{2}$ is defined by
$$T_{\phi} f = P(\phi f), \quad f \in H^{2}.$$
For $b$ in the unit ball of $H^{\infty}$, the de Branges-Rovnyak space $H(b)$ is the image of $H^{2}$ under the operator $(I-T_{b} T_{\overline{b}})^{1/2}$ with the inner product
$$\langle (I-T_{b} T_{\overline{b}})^{1/2} f, (I-T_{b} T_{\overline{b}})^{1/2} g\rangle = \langle f, g\rangle_{H^{2}}, \quad f, g \in [\ker(I-T_{b} T_{\overline{b}})^{1/2}]^{\perp}.$$
It is known that $H(b)$ space is a reproducing kernel Hilbert space with kernel
$$K_{\lambda}^{b}(z) = \frac{1-\overline{b(\lambda)}b(z)}{1-\overline{\lambda}z}, \quad \lambda, z \in \mathbb{D}.$$
So $\langle f, K_{\lambda}^{b}\rangle = f(\lambda)$ for all $f \in H(b)$ and $\lambda \in \mathbb{D}$. $H(b)$ spaces not only possess a beautiful structure, but also play an important role in many aspects of complex analysis and operator theory, most importantly, in the model theory for certain types of contractions, see the books of de Branges and Rovnyak \cite{dBR66}, Sarason \cite{Sa94} and the recent monographs of Fricain and Mashreghi \cite{FM16}.

Many properties of $H(b)$ depend on whether $b$ is or is not an extreme point of the unit ball of $H^{\infty}$. If $M_z$ is the forward shift operator on $H^{2}$, then $H(b)$ is invariant by $M_z$ if and only if $b$ is a nonextreme point of the unit ball of $H^{\infty}$ (\cite[P23]{Sa94}). Forward shift invariant subspaces are the natural subject to study, which yields profound research in operator theory, e.g. Beurling's theorem for the Hardy space (\cite{Be48}), Aleman-Richter-Sundberg's theorem for the Bergman space (\cite{ARS96}), Richter-Sundberg's theorems for the Dirichlet space (\cite{Ri88,RS92}). Despite the ongoing and active research on $H(b)$ spaces, we know little about the forward shift invariant subspaces of $H(b)$. In fact, the situation for $H(b)$ is difficult (\cite[Chap24, P352]{FM16}). In 2019, Aleman and Malman \cite{AlemanMalman} made great progress on this problem. They studied the de Branges-Rovnyak spaces $H(B)$ generated by Schur class functions $B$, and characterized the $M_z$-invariant subspace of $H(B)$ when $H(B)$ is of finite rank. In this paper, when $B = b$ is a rational function, we give another characterization of the $M_z$-invariant of $H(b)$, which provides some additional information of the structure of the $M_z$-invariant subspaces of $H(b)$.

In order to better understand operators on a Hilbert space, one may try to find models for certain classes, that is, a subclass of concrete operators with
the property that any given operator from the class is unitarily equivalent to an element of the subclass. It is known that the vector-valued de Branges-Rovnyak spaces $H(B)$ provide canonical model spaces for certain types of contractions. It is amazing that $H(B)$ spaces also provide model spaces for expansive analytic operators, see \cite[Proposition 2.1]{AlemanMalman}, also see \cite[Theorem 4.6]{LGR} for a detailed account. In this paper, we will show in detail that the scalar-valued de Branges-Rovnyak spaces $H(b)$ provide model spaces for what we call the expansive quasi-analytic operators $A$ with $A^*A -I$ being rank one, see definition \ref{quasianalytic} for the meaning of quasi-analytic operators.

The main results of this paper are the following two theorems.

\begin{theorem}\label{mainopmodel1}
Every quasi-analytic strict $2n$-isometry $A$ with $A^*A - I = w \otimes w$ is unitarily equivalent to $(M_z,H(b))$ for some rational function $b$ of degree $n$ with $b(0) = 0$.
\end{theorem}

When $b$ is a nonextreme point of the unit ball of $H^{\infty},$
then there exists a unique outer function $a$ such that $a(0)>0$ and
\begin{equation}
\left\vert a(z)\right\vert ^{2}+\left\vert b(z)\right\vert ^{2}=1,z\in
\mathbb{T}.\label{defa}%
\end{equation}
We call $a$ the pythagorean mate of $b$.
\begin{theorem}\label{invsinggtn1}
Let $T=(M_z,H(b))$ be a strict $2n$-isometry on $H(b)$. Then the pythagorean mate $a$ of $b$ has a single zero of multiplicity $n$ at some point $\overline{\lambda} \in \T$, and every nonzero $T$-invariant subspace is of the form $[(z-\overline{\lambda})^{j}\theta]_T$, where $j \in \{0,1,\cdots,n\}$, $\theta$ an inner function are such that $(z-\overline{\lambda})^{j}\theta \in H(b)$.
\end{theorem}
When we say $T=(M_z,H(b))$ is a $2n$-isometry on $H(b)$ we always mean that $M_z$ is a bounded operator on $H(b)$, and this happens if and only if $b$ is a nonextreme point of the unit ball of $H^\infty$.
In section 2, we introduce the background of $m$-isometries, and then prove Theorem \ref{mainopmodel1} using the rank one extension process. In section 3, we present the proof of Theorem \ref{invsinggtn1}.

\section{Operator model}

\subsection{M-isometry on de Branges-Rovnyak space}
A bounded linear operator $T$ on a Hilbert space $H$ is an
$m$-isometry for some positive integer $m$ if
\[
\beta_{m}(T):=\sum\limits_{k=0}^{m}(-1)^{m-k}\binom{m}{k}y^{k}x^{k}|_{y=T^{\ast},x=T}%
=\sum\limits_{k=0}^{m}(-1)^{m-k}\binom{m}{k}T^{\ast k}T^{k}=0.
\]
Note that
\[
\beta_{m+1}(T)=T^{\ast}\beta_{m}(T)T-\beta_{m}(T),
\]
So if $T$ is an $m$-isometry, then $T$ is an $n$-isometry for all $n\geq m.$
We say $T$ is a strict $m$-isometry if $T$ is an $m$-isometry but $T$ is not
an $(m-1)$-isometry.

The study of $m$-isometries originates from the work of Agler \cite{Ag90}, followed by a series of papers by Agler and Stankus \cite{AS956}. In 1991, Richter \cite{Ri91} investigated the cyclic analytic $2$-isometries and proved that any cyclic analytic $2$-isometry is unitarily equivalent to the shift operator $M_z$ on a Dirichlet type space $D(\mu)$ for some measure $\mu$ supported on the unit circle, where
$$D(\mu) = \left\{f \in \text{H}(\mathbb{D}): \int_{\mathbb{D}} |f'(z)|^{2} \int_{\mathbb{T}} \frac{1-|z|^{2}}{|z-\zeta|^{2}}d\mu(\zeta) dA(z) < \infty\right\}.$$
In 2014 and 2015, the first author Gu demonstrated some examples of $m$-isometric operators in \cite{Gu14, Gu15} and further studied the properties of $m$-isometries in \cite{Gu15b}. In 2018, the authors Gu and Luo \cite{GuLuo18} showed that for any positive integer $m$, the shift operator $M_z$ on $K_{m}$ is an $(m+2)$-isometry, where $K_m$ is the following space
\[
\left\{  f\in \text{H}({\mathbb{D)}}:\left\Vert
f\right\Vert ^{2}={\sum_{n\geq0}}\frac{\left(  n+1\right)
\cdots(n+m+1)}{(m+1)!}\left\vert a_{n}\right\vert ^{2}<\infty, f(z) = \sum_{n=0}^\infty a_n z^n\right\}  .
\]

In this paper, we are interested in the expansive $m$-isometry $T$ with $T^*T - I = w \otimes w$. The following result is proved in \cite[Corollaries 8.6 \text{ and } 8.11]{LGR}.
\begin{theorem}\label{thm1}
Assume $T$ is a bounded linear operator on a Hilbert space $H$ with $\Delta = T^{\ast}T-I = w\otimes w.$
\newline(i) The
operator $T$ is a $2n$-isometry if and only if there exists $\lambda
\in\mathbb{T}$ such that $\left(  T^{\ast}-\lambda\right)  ^{n}w$ $=0.$%
\newline(ii) If $\left(  T^{\ast}-\lambda\right)  ^{n}w=0$ and $\left(
T^{\ast}-\lambda\right)  ^{n-1}w\neq0,$ then $T$ is a strict $2n$%
-isometry.\newline(iii) The operator $T$ is a $(2n+1)$-isometry if and only if
$T$ is a $2n$-isometry.
\end{theorem}

Let $M_{z}$ be the forward shift operator on the Hardy space $H^2$, $M_zf(z)=zf(z)$, and $L$ be the backward shift operator,
\[
Lf(z)=\frac{f(z)-f(0)}{z}.
\]
Let $b$ be a nonextreme function in the unit ball of $H^{\infty}$. Then the de Branges-Rovnyak space $H(b)$ is invariant for both $M_z$ and
$L$ (\cite{FM16, Sa94}). Let $T = (M_z, H(b))$. Both $T$ and $L$ are bounded operators from $H(b)$ into $H(b)$. Note that $LT=I$ on $H(b)$, and $T$ is expansive and $L$ is contractive. We have
\[
L^{\ast}=T-b\otimes Lb,
\]
see e.g. \cite[Theorem 18.22]{FM16}. Therefore, $T=L^{\ast}+b\otimes Lb$ and $T^{\ast}=L+Lb\otimes b$.

\begin{lemma}\label{tstmiro}
Let $L$ and $T$ be defined as above. Then%
\[
T^{\ast}T-I=(1+\left\Vert b\right\Vert ^{2})Lb\otimes Lb.
\]

\end{lemma}

\begin{proof}
Note that%
\begin{align*}
T^{\ast}T  &  =\left(  L+Lb\otimes b\right)  T=LT+Lb\otimes T^{\ast}b\\
&  =I+Lb\otimes(Lb+\left\langle b,b\right\rangle Lb),
\end{align*}
where the inner product is in $H(b).$ The conclusion then follows.
\end{proof}

\subsection{Operator model}\label{operatormodel}
It is known that every cyclic analytic $2$-isometry $A$ on a Hilbert space $H$ is unitarily equivalent to the shift operator on some Dirichlet-type space $D(\mu)$, where $\mu$ is a positive measure on $\mathbb{T}$ (\cite{Al93, EKMR14, Ri91}). So the shift operator $M_z$ on $D(\mu)$ is an operator model for such operator $A$. In this section we will use the rank one extension process developed in \cite{Ri} to show that every quasi-analytic $2n$-isometry $A$ with $A^{\ast}A-I = w \otimes w$ is unitarily equivalent to the shift operator $M_z$ on $H(b)$ for some $b$.

If $T$ is an expansive operator, then we let $\Delta = T^*T - I$, and $[\ran \Delta]_{T^*}$ be the smallest $T^*$-invariant subspace containing the range of $\Delta$.
\begin{lemma}\label{decomsftnis}
Assume $T$ is a bounded linear operator on a Hilbert space $H$ with $\Delta = T^{\ast}T-I = w\otimes w$, and $T$ is a strict $2n$-isometry.
\newline(i) Let $H_{w}=\text{Span}\left\{
w,T^{\ast}w,\cdots,T^{\ast n-1}w\right\}  .$ Then $H_w = [w]_{T^*}$, $\dim H_w = n$ and $\sigma(T^{\ast}|H_{w})=\{\lambda\}$ for some $\lambda \in \T$.
\newline(ii) When $n=1$, $\{w\}^\perp$ is $T$-invariant, and $T_0 = T|\{w\}^\perp$ is an isometry.
\newline(iii) When $n \geq 2$, let $w_{n}=\left(  T^{\ast}-\lambda
\right)  ^{n-1}w$ and $\left\{  w_{n}\right\}  ^{\perp}=H\ominus \C\left\{
w_{n}\right\}  .$ Then $\left\{  w_{n}\right\}  ^{\perp}$ is $T$-invariant and $T_{n-1}=T|\left\{  w_{n}\right\}  ^{\perp}$ is a strict
$2(n-1)$-isometry such that%
\begin{equation*}
T_{n-1}^{\ast}T_{n-1}-I=w^{\prime}\otimes w^{\prime},
\end{equation*}
where $w^{\prime}:=P_{\left\{  w_{n}\right\}  ^{\perp}}w\in H_w  \ominus \C\left\{
w_{n}\right\}  .$
\end{lemma}
\begin{proof}
(i) This follows from Theorem \ref{thm1}.

(ii) When $n=1$, by part (i) $[w]_{T^*} = \C w$, so $\{w\}^\perp$ is $T$-invariant, and $T_0 = T|\{w\}^\perp$ is an isometry.

(iii) Since $(T^*-\lambda)w_{n} = 0$, we have $\{w_n\}^\perp$ is $T$-invariant. So
\begin{align*}
T_{n-1}^{\ast}T_{n-1}-I &= P_{\{w_n\}^\perp} (T^*T-I)P_{\{w_n\}^\perp}\\
& = P_{\{w_n\}^\perp} w \otimes P_{\{w_n\}^\perp} w\\
& = w' \otimes w'.
\end{align*}
By $w' = P_{\{w_n\}^\perp} w = w - \langle w, w_n\rangle\frac{w_n}{\|w_n\|^2}$, we obtain that $w' \in H_w  \ominus \C\left\{w_{n}\right\}$. Note that for $n\geq2$
\begin{align*}
(T_{n-1}^*-\lambda)^{n-1} w' &= P_{\{w_n\}^\perp} (T^*-\lambda)^{n-1} P_{\{w_n\}^\perp} w\\
& = P_{\{w_n\}^\perp} (T^*-\lambda)^{n-1} w = 0,
\end{align*}
and
\begin{align*}
(T_{n-1}^*-\lambda)^{n-2} w' =  P_{\{w_n\}^\perp} (T^*-\lambda)^{n-2} w \neq 0,
\end{align*}
we conclude from Theorem \ref{thm1} that $T_{n-1}$ is a strict $2(n-1)$-isometry.
\end{proof}

Note that if $\lambda \in \T$, then $\beta_m(\overline{\lambda}T) = \beta_m(T)$. So for a $2n$-isometry $T$ with $T^*T - I = w \otimes w, \left(  T^{\ast}-\lambda\right)  ^{n}w=0$, we may suppose $\lambda = 1$.
By the above lemma, there exists an orthonormal basis
$\left\{  e_{1},\cdots,e_{n}\right\}  $ of $H_{w}$ such that with respect to%
\[
H=H_{0}\oplus \C\left\{  e_{1}\right\}  \oplus\cdots\oplus\C\left\{
e_{n}\right\},  \text{ where }H_{0}=H\ominus H_{w} =H_w^\perp,
\]
the operator $T$ has the following matrix decomposition (we suppose $\lambda = 1$)%
\begin{align}
T&=\left[
\begin{array}
[c]{cccccc}%
V & h_{1}\otimes e_{1}  & \cdots & h_{n-1}\otimes e_{n-1}
& h_{n}\otimes e_{n}\\
0 & e_{1}\otimes e_{1}  & \cdots & a_{1,n-1}%
e_{1}\otimes e_{n-1} & a_{1n}e_{1}\otimes e_{n}\\
\vdots & \vdots  & \ddots & \ddots & \vdots\\
0 & 0  & \cdots & e_{n-1}\otimes e_{n-1} & a_{n-1,n}e_{n-1}\otimes e_{n}\\
0 & 0  & \cdots & 0 & e_{n}\otimes e_{n}%
\end{array}
\right]  \label{matrixT}
\end{align}
where $h_{i}\in H_{0}$ for $1\leq i\leq n$ and $a_{ij}\in\mathbb{C}$ for
$1\leq i<j\leq n$ and $V$ is an isometry on $H_{0}.$ Since $T^*T - I = w \otimes w$ is a rank one operator, one checks that $V^*h_i = 0, i =1, \ldots, n$.

To investigate the operator model for the above $2n$-isometry $T$, we introduce the following definition.
\begin{definition}\label{quasianalytic}
We say that an expansive operator $T$ is quasi-analytic if $T|\HM^\perp$ is unitarily equivalent to the unilateral shift $S = (M_z,H^2)$ acting on $H^2$, where $\HM = [\ran \Delta]_{T^*}, \Delta = T^*T - I$.
\end{definition}
Recall from \cite[Proposition 5.6]{AS956} that a $2$-isometry with $T^{\ast}T-I=w\otimes w$ is unitarily equivalent to $V\oplus \mathbf{B}$,
where $V$ is an isometry and $\mathbf{B}$ is a Brownian shift, i.e.
$$\mathbf{B}=\left[\begin{matrix}
S & \sigma(1\otimes1)\\
0 & \lambda%
\end{matrix} \right]
$$
on $H^{2}\oplus\mathbb{C}, \lambda \in \mathbb{T}$, $\sigma > 0$. Note that if $T^*T - I = w\otimes w$, then $[\ran \Delta]_{T^*} = H_w$. So if $T$ is a quasi-analytic $2$-isometry with $T^{\ast}T-I=w\otimes w$, then $T$ is unitarily equivalent to a Brownian shift. We also have a Brownian shift is unitarily equivalent to $(M_z,H(b))$ on $H(b)$ for some $b$, see Theorem \ref{mainopmodel}. Our definition of quasi-analyticity is partly motivated by this observation. It also turns out that if $T$ is a quasi-analytic $2n$-isometry with $T^*T - I = w\otimes w$ on $H$, then $T$ is analytic, i.e. $\bigcap_{n\geq 0} T^n H = \{0\}$. This result is not obvious. Our strategy is to use the rank one extension process to show that $T$ is unitarily equivalent to $(M_z,H(b))$ on $H(b)$ for some rational $b$ of degree $n$, thus obtaining Theorem \ref{mainopmodel1}. The analyticity of $T$ then follows from this.

We remark that Theorem \ref{mainopmodel1} can also be derived from \cite[Theorems 1.1, 1.3 and 4.6]{LGR}, but these theorems in \cite{LGR} apply to a wide class of operators and are not readily obtained. Therefore we provide a different approach here for just the $H(b)$ space setting.

Now we introduce the rank one extension process. Suppose $T$ is quasi-analytic with $T^*T - I = w\otimes w$, then we can assume
\[
T=\left[
\begin{matrix}
S & h\\
0 & A
\end{matrix}
\right]  \text{ on }H^{2}\oplus H_w%
\]
where  $h:H_w\to H^2$ with $\ran h \subseteq \ker S^*$ and $A:H_w\to H_w$ with $(A^*-1)^n=0$ (here we assume $\lambda = 1$).
By Lemma \ref{decomsftnis} and (\ref{matrixT}), one can construct every such $2n$-isometry $T$ inductively by a sequence of operators $T_j \in B(H^2\oplus \C^j)$, $j=0, \dots, n$ with the following properties:
\begin{itemize}
\item[(a1)] $T_0=S$ and for $0< j \le n$, $T_j$ is a strict $2j$-isometry with $T^*_jT_j-I=w_j \otimes w_j$ for some $w_j\in \{0\}\oplus \C^j$, and $(T^*-1)^jw_j = 0$,
\item[(a2)] for $0\le j<n$, $w_{j+1}=(w_j, \omega_{j+1})$,
\item[(a3)] for $0\le j<n$,
$$T_{j+1}=\left[\begin{matrix} T_j &B_j\\0&1\end{matrix}\right],$$
where $B_j:\C\to \ker S^*\oplus \C^j\subseteq H^2\oplus \C^j$ satisfies $T^*_jB_j1=\overline{\omega_{j+1}}w_j$.
\item[(a4)] for $0\le j<n$, $B^*_jB_j 1=|\omega_{j+1}|^2$,
\item[(a5)] $T=T_n, w=w_n$.
\end{itemize}

In fact, identifying $\C^j$ with $\text{Span }\{e_1, \dots , e_j\}$, we note that for $j=0, \dots ,n$  $H_j=H^2 \oplus \C^j$ is invariant for $T$, and the above conditions follow with $T_j=T|H_j$, $\omega_j=\langle w, e_j\rangle$, and $w_j=P_jw$, where $P_j$ denotes the projection onto $H_j$. The fact that
\begin{align*}T_{j+1}^*T_{j+1}-I_{H_{j+1}}&=P_{j+1}(T^*T-I)|H_{j+1}\\
&=w_{j+1}\otimes w_{j+1} \\
&= \left[\begin{matrix} w_j \otimes w_j &\overline{\omega_{j+1}}w_j \otimes e_{j+1}\\ \omega_{j+1}e_{j+1}\otimes w_j&|\omega_{j+1}|^2\end{matrix}\right]\\
\end{align*}
gives the identities $T^*_jB_j1=\overline{\omega_{j+1}}w_j$ and $B_j^*B_j 1=|\omega_{j+1}|^2$.

\begin{lemma}\label{kernel} Let $T_0,\dots, T_n$ be any sequence of operators satisfying $(a1)-(a5)$. Let $z\in \D$ and $j\ge 0$. Then the map $x \to \left[\begin{matrix} x\\-\frac{1}{1-\overline{z}}B_j^*x\end{matrix}\right]$ takes $\ker (T_j-z)^*$ onto $\ker (T_{j+1}-z)^*$.

In particular, we have that $\dim \ker T^*_j=1$ for each $j\ge 0$.
\end{lemma}
\begin{proof}Immediate.\end{proof}

By the above discussion, we consider operator of the following form $$R=\left[\begin{matrix} T &B\\0&1\end{matrix}\right]  \text{ on }H\oplus\C,$$ where $H$ is a Hilbert space, $\dim \ker T^*=1$.

\begin{lemma}\label{Bt} Let  $w=(w',\omega)\in H\oplus \C$ and $k\in \ker T^*$ with $\|k\|^2= {1+\|w'\|^2}$. Then $R^*R-I=w \otimes w$, if and only if
\begin{enumerate} \item $T^*T-I=w'\otimes w'$,
\item there is $t\in [0,2\pi)$ such that $B1=B_t1=\frac{\overline{\omega}}{1+\|w'\|^2}(e^{it}k+Tw')$.
\end{enumerate}
\end{lemma}
\begin{proof} One easily verifies that $R^*R-I=w \otimes w$, if and only if (i) is satisfied and $T^*B1= \overline{\om} w'$ and $B^*B1=|\omega|^2$.

Suppose that (i) and (ii) are satisfied. Then $\|Tw'\|^2=\langle T^*Tw',w'\rangle =\|w'\|^2+\|w'\|^4$ and since $k\perp Tw'$ we have
$B_t^*B_t1=\|B_t1\|^2=\frac{|\om|^2}{1+\|w'\|^2}+\frac{|\om|^2}{(1+\|w'\|^2)^2}\|Tw'\|^2$ $=|\om|^2$. Also,
$$T^*B_t1= \frac{\overline{\omega}}{1+\|w'\|^2}T^*Tw'= \frac{\overline{\omega}}{1+\|w'\|^2}(1+\|w'\|^2)w'=\overline{\om} w'.$$
Hence (i) and (ii) imply that $R^*R-I=w \otimes w$.

Conversely, assume that $R^*R-I=w \otimes w$. Then (i) is satisfied and hence $T$ is bounded below. Let $f\in \ker T^*$ be a unit vector. Since $\dim \ker T^*=1$ we must have $B1=\alpha f + Tx$ for some $\alpha\in \C$ and $x\in H$. We have to show that $|\alpha|^2=\frac{|\om|^2}{1+\|w'\|^2}$ and $x=\frac{\overline{\om}}{1+\|w'\|^2}w'$.

We have
\begin{align}\label{eqnxw}
\overline{\om}w'=T^*B1=T^*(\alpha f+Tx)=T^*Tx=x+\langle x, w'\rangle w'.
\end{align}
Hence $x= \overline{\om}w' - \langle x, w'\rangle w'$. Let $\beta = \overline{\om} - \langle x, w'\rangle$, then $x=\beta w'$. Plugging this into (\ref{eqnxw}) we obtain that $\overline{\om}=\beta(1+\|w'\|^2)$. Thus, $x$ has the required form and
\begin{align*}|\om|^2&=B^*B1=\|B1\|^2 = |\alpha|^2 + \|Tx\|^2\\&=|\alpha|^2 + |\beta|^2\|Tw'\|^2\\
&=|\alpha|^2+\frac{|\om|^2}{(1+\|w'\|^2)^2}(\|w'\|^2+\|w'\|^4) \\
&=|\alpha|^2+\frac{|\om|^2\|w'\|^2}{1+\|w'\|^2}
\end{align*}
Hence $|\alpha|^2=\frac{|\om|^2}{1+\|w'\|^2}$ and this concludes the proof.
\end{proof}
If $\om=0$, then $B=0$ and hence $R$ would have an isometric direct summand. In this case, if $(T^*-1)^n w' = 0$, then $(R^*-1)^n w = 0$.

\begin{lemma}\label{newkernel1}
If $w=(w',\omega)\in H\oplus \C$, $k\in \ker T^*$ with $\|k\|^2= {1+\|w'\|^2}$ and $$R_t=\left[\begin{matrix} T &B_t\\0&1\end{matrix}\right]$$  satisfies (i) and (ii) of  the previous lemma, then $$k_{t}=  \left[\begin{matrix}k\\ - \omega e^{-it}\end{matrix}\right]$$ satisfies $k_{t}\in \ker R_t^*$ with $\|k_{t}\|^2= {1+\|w\|^2} $.
\end{lemma}
\begin{proof} The condition on $\|k_{t}\|$ is obvious.
Using the definition of $B_t$ we obtain
$$B_t^*k= \langle k, B_t1\rangle =\frac{ \omega e^{-it}}{1+\|w'\|^2}\|k\|^2= \omega e^{-it}.$$
The fact that $k_{t}\in \ker R_t^*$ now follows from Lemma \ref{kernel} with $z=0$.
\end{proof}

\begin{lemma} \label{newkernel2} If $w=(w',\omega)\in H\oplus \C, \omega \neq 0$, $k\in \ker T^*$ with $\|k\|^2= {1+\|w'\|^2}$ and $$R_t=\left[\begin{matrix} T &B_t\\0&1\end{matrix}\right]$$  satisfies (i) and (ii) of  Lemma \ref{Bt}. In addition, we suppose $T$ is a strict $2n$-isometry with $(T^*-1)^n w' = 0$. Then
\begin{enumerate}
\item $(R_t^*-1)^{n+1} w = 0$.
\item There is a unique $t \in [0,2\pi)$ such that $(R_t^*-1)^{n} w = 0$.
\end{enumerate}
\end{lemma}
\begin{proof}
(i) We have
\begin{align*}
(R_t^*-1)^{n+1} w=\left[\begin{matrix} (T^*-1)^{n+1} &0\\B_t^*(T^*-1)^n&0\end{matrix}\right]\left[\begin{matrix} w'\\ \omega\end{matrix}\right] = 0.
\end{align*}
(ii) We have
\begin{align*}
(R_t^*-1)^{n} w&=\left[\begin{matrix} (T^*-1)^{n} &0\\B_t^*(T^*-1)^{n-1}&0\end{matrix}\right]\left[\begin{matrix} w'\\ \omega\end{matrix}\right] = \left[\begin{matrix} 0\\B_t^*(T^*-1)^{n-1}w'\end{matrix}\right].
\end{align*}
By Lemma \ref{Bt},
\begin{align*}
B_t^*(T^*-1)^{n-1}w'& = \langle B_t^*(T^*-1)^{n-1}w', 1\rangle\\
& = \langle (T^*-1)^{n-1}w', \frac{\overline{\omega}}{1+\|w'\|^2}(e^{it}k+Tw')\rangle\\
& = \frac{\omega}{1+\|w'\|^2} \langle (T^*-1)^{n-1}w', e^{it}k+w'\rangle.
\end{align*}
Note that using (\ref{matrixT}) we have $(T^*-1)^{n-1}w'$ is  a nonzero vector with 0 in the first $n-1$ entries, and Lemma \ref{newkernel1} ensures that the $n$-th entries of $k$ and $w'$ have the same modulus, it follows that there is a unique $t \in [0,2\pi)$ such that $B_t^*(T^*-1)^{n-1}w' = 0$. Hence there is a unique $t \in [0,2\pi)$ such that $(R_t^*-1)^{n} w = 0$.
\end{proof}
Part (ii) of the above lemma implies that if $T$ is a strict $2n$-isometry and one wants to extend $R_t$ to be a strict $2(n+1)$-isometry, then there is a unique $t \in [0,2\pi)$ that does not lead to a strict extension, since for this $t$, $(R_t^*-1)^{n} w = 0$ and $R_t$ is a $2n$-isometry by Theorem \ref{thm1}.

Now suppose that $H$ is a forward shift invariant reproducing kernel Hilbert space on $\D$ with kernel $k_w(z)$ such that $k_0(z)=1$ for all $z\in \D$. Also assume that with $T=(M_z,H)$ we have $T^*T=I + w'\otimes w'$ for some function $w'\in H$, and $\dim \ker T^*=1$. Set $b_0=\frac{-1}{\sqrt{1+\|w'\|^2}}Tw'$. Note that in this case the function $k$ as in Lemma \ref{Bt} is the constant function $k=\sqrt{1+\|w'\|^2}$, and hence the $B_t$ from Lemma \ref{Bt} is given by
$$B_t 1= \frac{\overline{\omega}}{\sqrt{1+\|w'\|^2}}(e^{it}-b_0).$$

\begin{lemma}\label{uniteqo}
Suppose $H$ satisfies the conditions as above. Let $t\in [0,2\pi)$, $0\ne \omega \in \C$,  and
$$R_t=\left[\begin{matrix} T &B_t\\0&1\end{matrix}\right]  \text{ on }H\oplus\C.$$
Then for each $z\in \D$ we have
$$ k_z^t=\left(\begin{matrix} k_z\\ \frac{\omega}{\overline{z}-1}\frac{1}{\sqrt{1+\|w'\|^2}}(e^{-it}-\overline{b_0(z)})\end{matrix}\right)\in \ker (R_t^*-\overline{z})$$
\end{lemma}
\begin{proof} For $z\in \D$ we have $k_z\in \ker(T^*-\overline{z})$. Hence according to Lemma \ref{kernel} we have to show that $\frac{1}{\overline{z}-1}B^*_tk_z$ equals the second component in the vector displayed in the lemma. We multiply through by $\overline{z}-1$ and compute
\begin{align*}
B^*_tk_z &= \La k_z,B_t1\Ra\\
&=\frac{{\omega}}{\sqrt{1+\|w'\|^2}}(e^{-it}\La k_z, 1\Ra -\La k_z, b_0\Ra)\\
&= \frac{{\omega}}{\sqrt{1+\|w'\|^2}}(e^{-it}- \overline{b_0(z)}).
\end{align*}
\end{proof}
Let $h_{t,\omega}(z)= \frac{\overline{\omega}}{\sqrt{1+\|w'\|^2}}\frac{1-e^{-it}b_0(z)}{z-1}$. If $\frac{1-e^{-it}b_0(z)}{z-1}\not\in H$, then for $z,w\in \D$  $$k^t_w(z)= k_w(z)+ \frac{|\omega|^2}{(\overline{w}-1)(z-1)}\frac{1}{1+\|w'\|^2}(1-e^{it}\overline{b_0(w)})(1-e^{-it}{b_0(z)})$$
is the reproducing kernel for a Hilbert function space $H_t$:
$$H_t=\{f+ \alpha h_{t,\omega}: f\in H, \alpha \in \C\}, \quad \|f+\alpha h\|^2=\|f\|^2+|\alpha|^2,$$
and $R_t$ in the above lemma is unitarily equivalent to $(M_z,H_t)$. We will show that $(M_z,H_t)$ is unitarily equivalent to $(M_z,H(b_t))$ for some $b_t$, which implies $R_t$ is unitarily equivalent to $(M_z,H(b_t))$.

Suppose $b$ is a nonextreme point of the unit ball of $H^{\infty}$. Let $|\alpha|<1$ and $b_\alpha(z)= \frac{b(z)-\alpha}{1-\overline{\alpha}b(z)}$. The following lemma is contained in \cite[Lemma 4.7]{LGR}.
\begin{lemma}\label{uniteqsohb}
For each $|\alpha|<1$, $(M_z,H(b))$ and $(M_z,H(b_\alpha))$ are unitarily equivalent.
\end{lemma}

Thus we may assume that $b(0)=0$. Now let $b_0$ be a nonextreme point, rational function with $\|b\|_\infty  \leq 1, b_0(0)=0$, and consider $H(b_0)$. Let $T = (M_z,H(b_0))$, $w' = -\sqrt{1+\|b_0\|^2}Lb_0$. Then
$$T^* T - I = (1+\|b_0\|^2)Lb_0 \otimes Lb_0 = w' \otimes w'.$$
Let $a$ be the pythagorean mate of $b_0$. Note that
$$
\begin{cases}
\|b_0\|^{2} = a(0)^{-2} - 1,\\
\|Lb_0\|^{2} = 1 - |b_0(0)|^{2} - a(0)^{2} = 1-a(0)^2,
\end{cases}
$$
see e.g. \cite[Corollary 23.9]{FM16}. So
\begin{align}\label{normofwib}
\|w'\|^2 = (1+\|b_0\|^2)\|Lb_0\|^2 = a(0)^{-2}-1 = \|b_0\|^2.
\end{align}
Since $TLb_0 = b_0$, it follows that $b_0 = \frac{-1}{\sqrt{1+\|w'\|^2}}Tw'$. Hence $H(b_0)$ satisfies the hypothesis of Lemma \ref{uniteqo}.

Set $s=\frac{|\omega|^2}{1+\|w'\|^2+|\omega|^2}$. Let $t\in [0,2\pi)$ be such that $e^{-it}b_0(1) \neq 1$, and consider
$$H(z)=H_{s,t}(z)=s\frac{1+z}{1-z} + (1-s) \frac{1+e^{-it}b_0(z)}{1-e^{-it}b_0(z)}.$$
Then Re$H(z) \ge 0$ and $H(0)=1$, hence there is $b_t=b_{s,t}$ in the unit ball of $H^\infty$ such that
\begin{align}\label{btdebbzas}
H(z)=\frac{1+b_t(z)}{1-b_t(z)} = s\frac{1+z}{1-z} + (1-s) \frac{1+e^{-it}b_0(z)}{1-e^{-it}b_0(z)}
\end{align}
One checks that $b_t$ is a rational function with $b_t(0)=0$, $b_t(1)=1$, and $b_t'(1)= 1/s$.

\begin{theorem}\label{unitequirt}
Let $b_0$ be a rational function with $\|b_0\|_\infty  \leq 1, b_0(0)=0$, $H = H(b_0)$, and $R_t$ be defined in Lemma \ref{uniteqo} and $b_t$ be defined by (\ref{btdebbzas}). Then $R_t$ is unitarily equivalent to $(M_z,H(b_t))$.
\end{theorem}
\begin{proof}
Note that
$$\frac{1-z\overline{w}}{(1-z)(1-\overline{w})}= \frac{1}{2}\left[\frac{1+z}{1-z} + \frac{1+\overline{w}}{1-\overline{w}}\right].$$
Thus
$$1-b_t(z)\overline{b_t(w)}=\frac{1}{2}(1-b_t(z))(1-\overline{b_t(w)})\left[\frac{1+b_t(z)}{1-b_t(z)} + \frac{1+\overline{b_t(w)}}{1-\overline{b_t(w)}}\right],$$
and
$$\frac{1-e^{it}b_0(z)\overline{e^{it}b_0(w)}}{(1-e^{it}b_0(z))(1-\overline{e^{it}b_0(w)})}= \frac{1}{2}\left[\frac{1+e^{it}b_0(z)}{1-e^{it}b_0(z)} + \frac{1+\overline{e^{it}b_0(w)}}{1-\overline{e^{it}b_0(w)}}\right].$$
So
\begin{align*} K^{b_t}_w(z)&= \frac{1-b_t(z)\overline{b_t(w)}}{1-z\overline{w}}\\
&= \frac{1}{2}\frac{(1-b_t(z))(1-\overline{b_t(w)})}{1-z\overline{w}}\left[H(z)+\overline{H(w)}\right]\\
&= e(z)\overline{e(w)} + f(z)\overline{f(w)} \frac{1-b_0(z)\overline{b_0(w)}}{1-z\overline{w}}\\
&=e(z)\overline{e(w)} + f(z)\overline{f(w)} K^{b_0}_w(z),
\end{align*}
where $e(z)= \sqrt{s}\frac{1-b_t(z)}{1-z}$ and $f(z)=\sqrt{1-s}\ \  \frac{1-b_t(z)}{1-e^{-it}b_0(z)}$.

Note that $s=\frac{|\omega|^2}{1+\|w'\|^2+|\omega|^2}$ and the reproducing kernel of $H_t$ is
\begin{align*}
&k^t_w(z)\\
&= K^{b_0}_w(z)+ \frac{|\omega|^2}{(\overline{w}-1)(z-1)}\frac{1}{1+\|w'\|^2}(1-e^{it}\overline{b_0(w)})(1-e^{-it}{b_0(z)})\\
&= \frac{1-\overline{b_0(w)}b_0(z)}{1-\overline{w}z} + \frac{|\omega|^2}{(\overline{w}-1)(z-1)}\frac{1}{1+\|w'\|^2}(1-e^{it}\overline{b_0(w)})(1-e^{-it}{b_0(z)}).
\end{align*}
So $K^{b_t}_w(z)=g(z)\overline{g(w)} k^t_w(z)$, where
$$g(z) = \frac{\sqrt{1+\|w'\|^2}}{\sqrt{1+\|w'\|^2+|\omega|^2}}\frac{1-b_t(z)}{1-e^{-it}b_0(z)}.$$
It follows that $H(b)= gH_t$ and $R_t$ is unitarily equivalent to $(M_z,H(b_t))$.
\end{proof}

Now we can prove Theorem \ref{mainopmodel1}. We restate it here.
\begin{theorem}\label{mainopmodel}
Every quasi-analytic strict $2n$-isometry $A$ with $A^*A - I = w \otimes w$ is unitarily equivalent to $(M_z,H(b))$ for some rational function $b$ of degree $n$ with $b(0) = 0$.
\end{theorem}
\begin{proof}
By the discussion before Lemma \ref{kernel}, we know that $A$ is obtained through a sequence of operators $T_0 = S, T_1, \cdots, T_n = A$ satisfying conditions $(a1)-(a5)$. Note that each $T_{j+1}$ is a quasi-analytic strict $2(j+1)$-isometry with $T_{j+1}^*T_{j+1} - I = w_{j+1} \otimes w_{j+1}$, and $T_{j+1}=\left[\begin{matrix} T_j &B_j\\0&1\end{matrix}\right]$ is a rank one extension of $T_j, 0 \leq j <n$. Thus if we show $T_{j+1}$ is unitarily equivalent to $(M_z,H(b))$ for some rational function $b$ of degree $(j+1)$ with $b(0) = 0$ under the condition that $T_{j}$ is unitarily equivalent to $(M_z,H(b_0))$ for some rational function $b_0$ of degree $j$ with $b_0(0) = 0$, then we are done.

By \cite[Proposition 5.6]{AS956} $T_1$ is unitarily equivalent to a Brownian shift
$$\mathbf{B}=\left[\begin{matrix}
S & \sigma(1\otimes1)\\
0 & 1%
\end{matrix} \right], \sigma > 0.
$$
Since $\mathbf{B}^{\ast}\mathbf{B}-I = \left[
\begin{matrix}
0 & 0\\
0 & \sigma^{2}%
\end{matrix}
\right]$, using (\ref{normofwib}) one checks that $\mathbf{B}$ is unitarily equivalent to $(M_z,H(b))$ with
$$b(z) = \frac{\gamma z}{1-\beta z}, a(z) = a(0) \frac{1-z}{1-\beta z}, \beta = \frac{1}{1+\sigma^2}, |\gamma| = \frac{\sigma^2}{1+\sigma^2},$$
This can also be deduced from (\ref{btdebbzas}) by setting $b_0 = 0, s = \frac{\sigma^2}{1+\sigma^2}$.
Thus we only need to consider $j \geq 1$, and without loss of generality, suppose $T_j = (M_z,H(b_0))$ for some rational function $b_0$ of degree $j$ with $b_0(0) = 0$. By Lemma \ref{Bt}, there is $t \in [0,2\pi)$ such that
$$T_{j+1} = R_t =\left[\begin{matrix} T_j &B_t\\0&1\end{matrix}\right]  \text{ on }H(b_0)\oplus\C.$$
By the remark below Lemma \ref{newkernel2}, there is a unique $t_0 \in [0,2\pi)$ such that $R_{t_0}$ is not a strict $2(j+1)$-isometry. We claim that $e^{-it_0}b_0(1) = 1$. In fact, if $e^{-it_0}b_0(1) \neq 1$, then the function $b_{t_0}$ defined by (\ref{btdebbzas}) is a rational function of degree $(j+1)$. Then by Theorem \ref{unitequirt}, $R_{t_0}$ is unitarily equivalent to $(M_z,H(b_{t_0}))$, which is a strict $2(j+1)$-isometry. This is a contradiction. Hence $e^{-it_0}b_0(1) = 1$. Since $T_{j+1}$ is a strict $2(j+1)$-isometry, there is $t \in [0,2\pi)\setminus{t_0}$ such that $T_{j+1} = R_{t}$. Then $e^{-it}b_0(1) \neq 1$ and $T_{j+1} = R_{t}$ is unitarily equivalent to $(M_z,H(b_{t}))$, where $b_{t}$ is defined by (\ref{btdebbzas}). The proof is complete.
\end{proof}

\section{Invariant subspaces of $M$-isometry}
In this section we study the invariant subspaces of the $2n$-isometric forward shift operator on the de Branges-Rovnyak space. There is a beautiful link between the strict $2n$-isometry $(M_z,H(b))$ and the higher order local Dirichlet integral $D_w^n(f), w\in \T, f \in H^2$, see \cite[Theorem 1]{LGR}. When $n =1$, $D_w^1(f)=D_w(f)$ is the local Dirichlet integral introduced in \cite{RS91}.
\subsection{Invariant subspaces of the $2n$-isometric shift operator}
Let $T=(M_z,H(b))$ be a strict $2n$-isometry on $H(b)$ and $a$ the pythagorean mate of $b$. The characterization of such $b$ is obtained in \cite{KZ15} for $n=1$, and in \cite{LGR} for general $n$. By \cite[Theorem 1.1]{LGR}, $b$ is a rational function of degree $n$ such that the mate $a$ has a single zero of multiplicity $n$ at some point $\overline{\lambda} \in \T$.
Since $b$ are rational
functions of order $n$ with $\left\Vert b\right\Vert _{\infty}\leq1$, we have $M(\overline{a}) = H(b)$ (\cite{Sa86}), where $M(\overline{a}) = T_{\overline{a}}H^2$.
Thus all the $H(b)$ spaces are equal with equivalent norms. Let $e_n(z) = \frac{(1-\lambda z)^{n}}{2^{n}}$. Since $e_{n}/a$ and
$a/e_{n}$ are bounded on $\mathbb{T}$, we also have $M(\overline{e_{n}}) =
M(\overline{a}) = H(b)$ (\cite{FHR16}). Let $S_{n} = (M_z,M(\overline
{e_n}))$. Then the $T$-invariant subspaces of $H(b)$ are just the $S_n$-invariant subspaces of $M(\overline{e_{n}})$.  When $n =1, \lambda= 1$, the $S_{1}$-invariant subspaces
of $M(\overline{e_{1}})$ were characterized in \cite{Sa86}. In this case,
$M(\overline{e_{1}}) = M(e_1)H^2 \dotplus\mathbb{C} = D(\delta_{1})$
(\cite{RS91}), and the $S_{1}$-invariant subspaces of
$M(\overline{e_{1}})$ are $\{0\}, M(e_{1}) \cap M(\theta), M(\overline{e_{1}})
\cap M(\theta)$ with $\theta$ an inner function (\cite[Theorem 7]{Sa86}). We will show that a similar characterization exists for all $n \geq 1$.

By the above discussion, we may suppose $a(z) = \frac{(1- \lambda z)^{n}}{2^{n}}$, and $b$ is the pythagorean mate of $a$. Then by Fej\'{e}r-Riesz theorem (\cite{RN90}), $b$ is a polynomial of degree $n$. For simplicity, we take $\lambda = 1$ in the following. It was shown in \cite{FM08} that every function
in $H(b)$, as well as its derivative up to order $n-1$ has a finite
nontangential limit at $1$. For $w \in\mathbb{D}, i \in\mathbb{N}$, we define
\[
u_{w}^{i}(z) = \frac{\partial^{i}}{\partial\overline{w}^{i}}K_{w}^{b}(z) =
\frac{\partial^{i}}{\partial\overline{w}^{i}} \left( \frac{1-\overline
{b(w)}b(z)}{1-\overline{w}z} \right) , \quad z \in\mathbb{D},
\]
and
\begin{align*}
u_{1}^{i}(z) = \angle\lim\limits_{w \rightarrow1} u_{w}^{i}(z), \quad z
\in\mathbb{D}%
\end{align*}
the nontangential limit as $w \rightarrow 1$. Let $\mathcal{P}_{n}$ be
the space of all polynomials with degree less than or equal to $n$. By \cite[Theorems 1.3 \&
1.4]{FHR16}, see also \cite[Theorem 7.2]{LGR}, we have
\[
H(b) = M(a) \oplus\text{\text{Span}} \{u_{1}^{i}: 0 \leq i \leq n-1\}= M(a)
\oplus\mathcal{P}_{n-1},
\]
where the orthogonal decomposition is with respect to the inner product in
$H(b)$. Let
\[
\mathcal{L}_{0} = \text{\text{Span}} \{u_{1}^{i}: 0 \leq i \leq n-1\},
\]
and
\begin{align*}
\mathcal{L}_{j} = \mathcal{L}_{0} \cap\left( \text{\text{Span}} \{u_{1}^{i}: 0
\leq i < j\}\right) ^{\perp}, \quad1\leq j \leq n-1.
\end{align*}
Since $\langle f, u_1^i\rangle = f^{(i)}(1)$, it follows that
\begin{align}\label{speceLj}
\mathcal{L}_{j} = \text{\text{Span}}\{(z-1)^{j}, (z-1)^{j+1}, \ldots, (z-1)^{n-1}%
\},\quad0 \leq j \leq n-1.
\end{align}

Let
$$\text{Mult}(H(b)) = \{\varphi \in H(b): \varphi f \in H(b), \forall f \in H(b)\}$$
be the multiplier algebra of $H(b)$. The following lemma is contained in \cite{FHR19}, we include a proof here for completeness.
\begin{lemma}
\label{multohb} Let $a(z) = \frac{(1- z)^{n}}{2^{n}}$ and $b$ the pythagorean mate of $a$. Then $\text{Mult}(H(b))$ $= H(b)
\cap H^{\infty}$.
\end{lemma}

\begin{proof}
Since $H(b)$ is a reproducing kernel Hilbert space, we have $\text{Mult}(H(b))$
$\subseteq H(b) \cap H^{\infty}$. Conversely, let $\varphi\in H(b) \cap H^{\infty}$.
Let $f \in H(b)$, there are $g \in H^{2}, p \in\mathcal{P}_{n-1}$ such
that $f = a g + p$. So
\[
\varphi f = a (\varphi g) + p \varphi.
\]
Since $H(b)$ is $T$-invariant and $\varphi\in H(b) \cap H^{\infty}$, we obtain
$p \varphi\in H(b)$ and $\varphi g \in H^{2}$. Hence $\varphi f \in H(b)$ and
$\varphi\in\text{Mult}(H(b))$. The proof is complete.
\end{proof}

Let $\theta$ be an inner function. Then $M(a) \cap M(\theta)$ is a $T$-invariant subspace of $H(b)$. It is not hard to verify that $M(a) \oplus\mathcal{L}_{j},
0 \leq j \leq n-1,$ are $T$-invariant subspaces of $H(b)$, or see the following Lemma \ref{intersection}. If furthermore, $\theta\left( M(a) \oplus\mathcal{L}_{j}\right)  \subseteq H(b)$ for some $j \in \{0, \ldots, n-1\}$, then since $T_{\overline{\theta}}$ is a contraction on $H(b)$, we have $\theta\left( M(a) \oplus\mathcal{L}_{j}\right)$ is closed, and thus a $T$-invariant subspace of $H(b)$.
Now we can describe the $T$-invariant subspaces of $H(b)$.
\begin{theorem}\label{invsubsftni}
Let $a, b$ be as in the above lemma. Then the $T$-invariant subspaces of $H(b)$ are $\{0\}$, $M(a) \cap M(\theta)$ with $\theta$ an inner function, and $\theta\left( M(a) \oplus\mathcal{L}_{j}\right)$, where $0 \leq j \leq n-1$, $\theta$ an inner function are such that $\theta\left( M(a) \oplus\mathcal{L}_{j}\right)  \subseteq H(b)$.
\end{theorem}

\begin{proof}
Let $\mathcal{N}$ be a nonzero $T$-invariant subspace of $H(b)$. We split the
proof in the following three cases.

$\mathbf{Case 1}$. $\mathcal{N}$ is contained in $M(a)$. Since $T_{a}: H^{2}
\rightarrow M(a)$ defined by $T_{a} f = af$ is a unitary operator. By
Beurling's theorem, $\mathcal{N} = a\theta H^{2} = M(a) \cap M(\theta)$ for
some inner function $\theta$.

$\mathbf{Case 2}$. $\mathcal{N}$ is not contained in $M(a)$ and the greatest
common inner divisor of the functions in $\mathcal{N}$ is $1$. By case 1,
$\mathcal{N} \cap M(a) = M(a) \cap M(\theta)$ for some inner function $\theta
$. Note that $\forall f \in\mathcal{N}$, we have $a f \in\mathcal{N} \cap
M(a)$ and $af = \theta g$ for some $g \in H^{2}$. Since $a$ is an outer
function, we obtain that $\theta$ is an inner divisor of $f$. This is true for
all functions in $\mathcal{N}$, we then get that $\theta= 1$, and $\mathcal{N}
\cap M(a) = M(a)$, $M(a) \subsetneq\mathcal{N}$. Note that $M(a) \oplus\mathcal{L}_{j}, 0 \leq j \leq n-1$
are $n$ $T$-invariant subspaces containing $M(a)$ and there is no other
$T$-invariant subspace containing $M(a)$. So $\mathcal{N} = M(a)
\oplus\mathcal{L}_{j}$ for some $j \in\{0, 1, \cdots, n-1\}$.

$\mathbf{Case 3}$. $\mathcal{N}$ is not contained in $M(a)$ and the greatest
common inner divisor $\theta$ of the functions in $\mathcal{N}$ is not $1$. By
case 1, $\mathcal{N} \cap M(a) = M(a) \cap M(u)$ for some inner function $u$.
By the argument in case $2$, we see that $u$ is an inner divisor of $\theta$.
Let $au \in M(a) \cap M(u)$, then $au \in\mathcal{N}$ and $\theta$ is an inner
divisor of $au$. So $\theta$ is a divisor of $u$ and $\mathcal{N} \cap M(a) =
M(a) \cap M(\theta)$. It follows that $a \theta H^{2} \subseteq\mathcal{N}$.
Let $\mathcal{R} = T_{\overline{\theta}} \mathcal{N}$. It is not hard to check
that $\mathcal{N} = \theta\mathcal{R}$. Hence $T \mathcal{R} \subseteq
\mathcal{R}$ and
\[
M(a) = T_{\overline{\theta}} (a \theta H^{2}) \subseteq T_{\overline{\theta}}
\mathcal{N} = \mathcal{R}.
\]
Next we show $\mathcal{R}$ is closed. In fact, note that $M(a)$ is a subspace which is closed and of finite co-dimension in $H(b)$, thus
$\mathcal{R}$ is closed, as it is the pre-image of a closed subspace $\mathcal{R}/M(a)$ under the quotient mapping $H(b) \rightarrow H(b) / M(a)$.
Therefore $\mathcal{R}$ is a $T$-invariant subspace containing $M(a)$, and then $\mathcal{R} = M(a) \oplus\mathcal{L}_{j}$ for some $j \in \{0, \ldots, n-1\}$.
\end{proof}

In fact, if $\theta\left( M(a) \oplus\mathcal{L}_{j}\right)  \subseteq H(b)$ for some $0 \leq j \leq n-1$, we will show that
\begin{align*}
\left( M(a) \oplus\mathcal{L}_{k}\right)  \cap M(\theta) = \theta\left( M(a)
\oplus\mathcal{L}_{k}\right),  \quad j \leq k \leq n-1,
\end{align*}
and thus obtaining a similar characterization for the $T$-invariant subspaces of $H(b)$ as in \cite[Theorem 7]{Sa86}. To achieve this, we need the following lemmas.

\begin{lemma}
\label{invssfis} Let $\theta$ be an inner function, and $a, b$ be as in Lemma \ref{multohb}. Then $M(a) \cap M(\theta)$ and $\left( M(a)
\oplus\mathcal{L}_{j}\right)  \cap M(\theta)$, $0 \leq j \leq n-1,$ are
$T$-invariant subspaces of $H(b)$.
\end{lemma}
\begin{proof}
It is clear that $M(a) \cap M(\theta)$ and $\left( M(a) \oplus\mathcal{L}%
_{j}\right)  \cap M(\theta), 0 \leq j \leq n-1,$ are closed subspaces and
$M(a) \cap M(\theta)$ is a $T$-invariant subspace. Next we show $\left( M(a)
\oplus\mathcal{L}_{j}\right)  \cap M(\theta), 0 \leq j \leq n-1,$ are
$T$-invariant. Let $f \in H(b)$, we have
\begin{align*}
\langle f, T^{*} u_{1}^{i}\rangle= \langle zf, u_{1}^{i}\rangle= if^{(i-1)}(1)
+ f^{(i)}(1) = \langle f, u_{1}^{i}+iu_{1}^{i-1}\rangle.
\end{align*}
So $T^{*} u_{1}^{i} = u_{1}^{i}+iu_{1}^{i-1}$ and $M(a) \oplus\mathcal{L}_{j},
0 \leq j \leq n-1,$ are $T$-invariant. Thus $\left( M(a) \oplus\mathcal{L}%
_{j}\right)  \cap M(\theta), 0 \leq j \leq n-1,$ are $T$-invariant.
\end{proof}

\begin{lemma}\label{intersection}
Let $\theta$ be an inner function, and $a, b$ be as in the above lemma. Then for $0\leq j \leq n-1$,
\begin{align*}
\left( M(a) \oplus\mathcal{L}_{j}\right)  \cap M(\theta) = M(a) \cap M(\theta) + \left( M(a) \oplus
\mathcal{L}_{j}\right)  \cap\left( \theta\mathcal{P}_{n-1}\right).
\end{align*}
\end{lemma}
\begin{proof}
We only need to show that the space on the left is contained in the space on the right. Let $f \in \left( M(a) \oplus\mathcal{L}_{j}\right)  \cap M(\theta)$, then $f \in H(b)\cap M(\theta)$. Suppose $f = \theta g, g \in H^2$. Since $T_{\overline{\theta}}$ is a contraction on $H(b)$, we have $g \in H(b)$. Then there are $h \in H^2, p \in \HP_{n-1}$ such that $g = a h + p$. So $f = a\theta h + \theta p$, which belongs to the space on the right.
\end{proof}
Recall that $D(\delta_1) = (z-1)H^{2} \dotplus\mathbb{C}$ (\cite{RS91}). If $\theta$ is an inner function, then $\theta \in D(\delta_1)$, if and only if $\theta$ has a finite angular derivative at $1$ (\cite{Sa86, RS91}).
\begin{lemma}\label{thetadelta}
Let $a, b$ be as in the above lemma. Suppose $\theta$ is an inner function and $\theta \not \in (z-1)H^{2} \dotplus\mathbb{C}$. Then
$$\left( M(a) \oplus\mathcal{L}_{j}\right)  \cap M(\theta) = M(a) \cap M(\theta), \quad 0\leq j \leq n-1.$$
\end{lemma}

\begin{proof}
We first show that $\theta\mathcal{P}_{n-1} \cap H(b) = \{0\}$. If this holds, then
$$\left( M(a) \oplus \mathcal{L}_{j}\right)  \cap\left( \theta\mathcal{P}_{n-1}\right)\subseteq \theta\mathcal{P}_{n-1} \cap H(b) = \{0\}.$$
Lemma \ref{intersection} then ensures that the conclusion holds.

Now we prove $\theta\mathcal{P}_{n-1} \cap H(b) = \{0\}$. Note that $H(b) = (z-1)^n H^2 \dotplus \HP_{n-1} \subseteq D(\delta_1)$, and
$$\HP_{n-1} = \text{Span}\{1, (z-1), \ldots, (z-1)^{n-1}\} = \mathcal{L}_{0}.$$
If $n =1$, then since $\mathcal{P}_{0} = \C, H(b) = (z-1)H^{2} \dotplus\mathbb{C}$, we have $\theta\mathcal{P}_{n-1} \cap H(b) = \{0\}$. If $n \geq 2$, we proceed as follows.

Let $f \in \theta\mathcal{P}_{n-1} \cap H(b)$, then there are $c_0, \ldots, c_{n-1} \in \C$ such that
$$f =\sum_{i=0}^{n-1} c_i(z-1)^i \theta \in H(b) \subseteq D(\delta_1).$$
Since for $i \geq 1$, $(z-1)^i \theta \in (z-1)H^2 \subseteq D(\delta_1)$, we obtain that $c_0 \theta \in D(\delta_1)$, and so $c_0 = 0$. Then $f = \sum_{i=1}^{n-1} c_i(z-1)^i \theta \in H(b) = (z-1)^n H^2 \dotplus \HP_{n-1}$.
Suppose
$$f = \sum_{i=1}^{n-1} c_i(z-1)^i \theta = (z-1)^n g + p, \quad g \in H^{2}, p \in\mathcal{P}_{n-1}.$$
Then $p(1) = 0$, and
$$\sum_{i=1}^{n-1} c_i(z-1)^{i-1} \theta = (z-1)^{n-1} g + \frac{p}{z-1}.$$
Note that $(z-1)^{n-1} g + \frac{p}{z-1} \in (z-1)H^2  \dotplus\mathbb{C} = D(\delta_1)$, applying the same argument, we obtain that $c_1 = 0$. Similarly, we have $c_2 = 0, \ldots, c_{n-1} = 0$, and $f = 0$. The proof is complete.
\end{proof}
%If $\theta \in (z-1)H^{2} \dotplus\mathbb{C}$, then $(z-1)^{n-1}\theta \in M(a) \oplus\mathcal{L}_{n-1}$. Thus it follow from the above lemma that $M(a) \cap M(\theta) = \left( M(a) \oplus\mathcal{L}_{j}\right)  \cap M(\theta)$ for some $j \in \{0, \ldots, n-1\}$, if and only if $\theta \not \in (z-1)H^{2} \dotplus\mathbb{C}$.

\begin{lemma}\label{invssfis2}
Let $\theta$ be an inner function, and $a, b$ be as in the above lemma. If $\theta\left( M(a) \oplus\mathcal{L}%
_{j}\right)  \subseteq H(b)$ for some $0 \leq j \leq n-1$, then
\begin{align*}
\left( M(a) \oplus\mathcal{L}_{k}\right)  \cap M(\theta) = \theta\left( M(a)\oplus\mathcal{L}_{k}\right),  \quad j \leq k \leq n-1.
\end{align*}

\end{lemma}

\begin{proof}
If $\theta\left( M(a) \oplus\mathcal{L}_{j}\right)  \subseteq H(b)$, then $\theta(z-1)^j \in H(b)=(z-1)^n H^2 \dotplus \HP_{n-1}$. So $\theta \in (z-1)^{n-j}H^2 \dotplus \HP_{n-j-1}$.

If $n=1$, then $j =0$, $\theta \in (z-1)H^2 \dotplus \C = H(b)$. So $\theta$ is a multiplier of $H(b) = M(a) \oplus\mathcal{L}_{0} = M(a) \oplus \C$.
By Lemma \ref{intersection},
\begin{align*}
\left( M(a) \oplus\mathcal{L}_{0}\right)  \cap M(\theta)& = M(a) \cap M(\theta) + \C \theta\\
&= \theta M(a) + \theta \mathcal{L}_{0} = \theta\left( M(a)\oplus\mathcal{L}_{0}\right).
\end{align*}

If $n \geq 2$, we have the following cases.

$\mathbf{Case 1}$. $\theta\in\left( (z-1)^{i-1}H^{2} \dotplus\mathcal{P}_{i-2}\right)  \setminus\left( (z-1)^{i} H^{2} \dotplus\mathcal{P}_{i-1}\right) , 2 \leq i \leq n$.

Suppose $\theta = (z-1)^{i-1} g +p, g \in H^2, p \in \mathcal{P}_{i-2}$. By the same argument as in Lemma \ref{thetadelta}, we have $(z-1)^k\theta  \not\in H(b), 0 \leq k \leq n-i$, and
\begin{align*}
\theta\mathcal{P}_{n-1} \cap H(b) &  = \text{\text{Span}}\{(z-1)^{n-i+1}\theta, (z-1)^{n-j+2}\theta, \ldots, (z-1)^{n-1} \theta\}\\
&  = \theta\mathcal{L}_{n-i+1}.
\end{align*}
Now we show
$$\left( M(a) \oplus \mathcal{L}_{k}\right)  \cap\left( \theta\mathcal{P}_{n-1}\right)=\theta \mathcal{L}_{k}, \quad n-i+1 \leq k \leq n-1. $$
Note that
\begin{align*}
\left( M(a) \oplus \mathcal{L}_{k}\right)  \cap\left( \theta\mathcal{P}_{n-1}\right)& = \left( M(a) \oplus \mathcal{L}_{k}\right)  \cap\left( \theta\mathcal{P}_{n-1}\right) \cap H(b)\\
& = \left( M(a) \oplus \mathcal{L}_{k}\right)  \cap \theta\mathcal{L}_{n-i+1}.
\end{align*}
It is not hard to see that $\theta \mathcal{L}_k \subseteq \left( M(a) \oplus \mathcal{L}_{k}\right), n-i+1 \leq k \leq n-1$. Since $\theta \in D(\delta_1)$, we have $\theta(1) \neq 0$. Then in the expression of $\theta = (z-1)^{i-1} g +p, g \in H^2, p \in \mathcal{P}_{i-2}$, we have $p(1) \neq 0$. It follows that if $0 \leq l < k$, then $(z-1)^l\theta \not \in \left( M(a) \oplus \mathcal{L}_{k}\right)$. Thus arguing as in Lemma \ref{thetadelta}, we conclude that
\begin{align*}
\left( M(a) \oplus \mathcal{L}_{k}\right)  \cap\left( \theta\mathcal{P}_{n-1}\right)& = \left( M(a) \oplus \mathcal{L}_{k}\right)  \cap \theta\mathcal{L}_{n-i+1}\\
& = \theta \mathcal{L}_{k}, \quad n-i+1 \leq k \leq n-1.
\end{align*}
Lemma \ref{intersection} then implies
\begin{align*}
\left( M(a) \oplus\mathcal{L}_{k}\right)  \cap M(\theta)& = M(a) \cap M(\theta) + \theta \mathcal{L}_{k}\\
&= \theta\left( M(a)\oplus\mathcal{L}_{k}\right), \quad n-i+1 \leq k \leq n-1.
\end{align*}

$\mathbf{Case 2}$. $\theta\in H(b)$. In this case $\theta$ is a multiplier of $H(b)$. By using the argument in Case 1, we have
$\theta \mathcal{L}_{k} \subseteq M(a) \oplus\mathcal{L}_{k}, 0 \leq k \leq n-1$,
and for $k \geq 1$, $(z-1)^i \theta \not \in M(a) \oplus\mathcal{L}_{k}, 0 \leq i < k$.
Hence
$$\left( M(a) \oplus \mathcal{L}_{k}\right)  \cap\left( \theta\mathcal{P}_{n-1}\right)=\theta \mathcal{L}_{k}, \quad 0 \leq k \leq n-1,$$
and
\begin{align*}
\left( M(a) \oplus\mathcal{L}_{k}\right)  \cap M(\theta) = \theta\left( M(a)\oplus\mathcal{L}_{k}\right), \quad 0 \leq k \leq n-1.
\end{align*}
The proof is complete.
\end{proof}
Combining Theorem \ref{invsubsftni} and Lemma \ref{invssfis2}, we obtain the following similar result as in \cite[Theorem 7]{Sa86}.
\begin{theorem}\label{invsubsftninew}
Let $a, b$ be defined as in the above lemma. Then the $T$-invariant subspaces of $H(b)$ are $\{0\}$, $M(a) \cap M(\theta)$ and $\left( M(a) \oplus\mathcal{L}_{j}\right)  \cap M(\theta)$, $0 \leq j \leq n-1,$ with $\theta$ an inner function.
\end{theorem}

\

For $f \in H(b)$, Let $[f]_T$ be the closure of the polynomial multiples of $f$. In the following, we write $[f]$ for $[f]_T$. If $[f] = H(b)$, then $f$ is called a cyclic function. We prove the following version of Theorem \ref{invsinggtn1}.
\begin{theorem}\label{invsinggtn}
Let $T = (M_z,H(b))$ be a strict $2n$-isometry on $H(b)$. Suppose the pythagorean mate $a$ of $b$ has a single zero of multiplicity $n$ at $\overline{\lambda} \in \T$. Then every nonzero $T$-invariant subspace is of the form $[(z-\overline{\lambda})^{j}\theta]$, where $j \in \{0,1,\ldots,n\}$, $\theta$ an inner function are such that $(z-\overline{\lambda})^{j}\theta \in H(b)$.
\end{theorem}
\begin{proof}
Let $a(z) = \frac{(1- \lambda z)^{n}}{2^{n}}$ and $b_1$ the pythagorean mate of $a$, and $T_1 = (M_z,H(b_1))$. By the discussion before Lemma \ref{multohb}, we have the $T$-invariant subspaces of $H(b)$ are just the $T_1$-invariant subspaces of $H(b_1)$.

Let $\mathcal{N}$ be a nonzero $T$-invariant subspace of $H(b)$. Theorem \ref{invsubsftni} then implies that $\mathcal{N} = M(a) \cap M(\theta)$, or $\mathcal{N} = \theta\left(M(a) \dotplus\mathcal{L}_{j}\right) $ for some $0 \leq j \leq n-1$ and $\theta$ an inner function. If $\mathcal{N} = M(a) \cap M(\theta)$, then
$$\mathcal{N} = a\theta H^2 = [(z-\overline{\lambda})^{n}\theta].$$
If $\mathcal{N} = \theta\left(M(a) \dotplus\mathcal{L}_{j}\right) $ for some $0 \leq j \leq n-1$, then $\theta \mathcal{L}_{j} \in H(b)$. Since
\[
\mathcal{L}_{j} = \text{\text{Span}}\{(z-\overline{\lambda})^{j}, (z-\overline{\lambda})^{j+1}, \ldots, (z-\overline{\lambda})^{n-1}%
\},\quad0 \leq j \leq n-1,
\]
we have $(z-\overline{\lambda})^{j} \theta \in H(b)$. So $(z-\overline{\lambda})^{j} \theta \in \text{Mult}(H(b))$ and $\mathcal{N} = [(z-\overline{\lambda})^{j} \theta]$.
\end{proof}

When $b = \frac{z+1}{2}$, the following result was obtained in \cite{FMS15}.
\begin{corollary}\label{cyclicityhb}
Let $T=(M_z,H(b))$ be a strict $2n$-isometry on $H(b)$. Suppose the pythagorean mate $a$ of $b$ has a single zero of multiplicity $n$ at $\overline{\lambda} \in \T$. Let $f \in H(b)$, then $f$ is cyclic if and only if $f$ is outer and $f(\overline{\lambda}) \neq 0$.
\end{corollary}
\begin{proof}
The necessity is clear since $H(b) \subseteq H^2$ and the point evaluation at $\overline{\lambda}$ is continuous on $H(b)$. For sufficiency, let $f \in H(b)$, if $f$ is outer and $f(\overline{\lambda}) \neq 0$, then by the above theorem, there are $j \in \{0,1,\cdots,n\}$, $\theta$ inner function such that
$$[f] = [(z-\overline{\lambda})^{j} \theta].$$
So $\theta = 1, j =0$, and $[f] = [1] = H(b)$. The proof is complete.
\end{proof}

\subsection{The case of the rational functions}
Let $b$ be a rational function of degree $n$ that is a nonextreme point of the unit ball of $H^\infty$. By Fej\'{e}r-Riesz theorem (\cite{RN90}), there are $\lambda_1, \cdots, \lambda_n \in \overline{\D}$ (may not be distinct) such that the numerator of the Pythagorean mate $a$ of $b$ is a constant multiple of $\prod_{i=1}^n (1-\lambda_i z)$. Without loss of generality, suppose $\lambda_i \in \T$ for $1\leq i \leq m, 0 \leq m \leq n$. Then
$$H(b) = \prod_{i=1}^m (z-\overline{\lambda_i})H^2 \dotplus \mathcal{P}_{m-1}.$$
If there is no $\lambda_i \in \T$, then $m = 0$ and $H(b) = H^2$, see \cite[Lemma 4.3]{CR13}. Let $a_1(z) = \frac{1}{2^m}\prod_{i=1}^m (1-\lambda_iz)$, then as discussed in the above section 3.1, $H(b)= M(\overline{a_1})$ and the $T$-invariant subspaces of $H(b)$ are just the forward shift invariant subspaces of $M(\overline{a_1})$. Using the arguments in section 2.1, similarly, one can characterize the $T$-invariant subspaces of $H(b)$. Thus the proof of the following theorem is similar as in section 2.1, we omit the details.
\begin{theorem}\label{invrational}
Let $b$ be a rational function of degree $n$ that is a nonextreme point of the unit ball of $H^\infty$. Then there is $a(z) = \frac{1}{2^m}\prod_{i=1}^m (z-\lambda_i)$ such that $H(b) = M(\overline{a}), \lambda_i \in \T, 1 \leq i \leq m, 0 \leq m \leq n$. When $m = 0$, $H(b) = H^2$. So every nonzero $T$-invariant subspace of $H(b)$ is of the form $\left[\prod_{j=1}^k(z-\lambda_{i_{j}})\theta\right]$, where $i_{j} \in \{1,2,\cdots,m\}$ distinct, $0\leq k \leq m, j = 1, 2, \cdots, k$, $\theta$ inner function are such that $\prod_{j=1}^k(z-\lambda_{i_{j}})\theta \in H(b)$. When $k = 0$, $\left[\prod_{j=1}^k(z-\lambda_{i_{j}})\theta\right]$ is interpreted as $[\theta]$.
\end{theorem}

It follows from the above theorem that $f \in H(b)$ is cyclic if and only if $f$ is outer and $f(\lambda_i) \neq 0, 1 \leq i \leq m$.

In \cite{AlemanMalman}, Aleman and Malman studied the vector-valued $H(B)$ spaces. When $H(B)$ is of finite rank and $M_z$-invariant, they proved in \cite[Theorem 1.3]{AlemanMalman} that for $M_z$-invariant subspace $\HM$, $\dim(\HM\ominus M_z\HM) = 1$. Moreover, if $\varphi \in \HM\ominus M_z\HM$ is of norm $1$, then $\HM = [\varphi]_{M_z}$ and there is a Schur function $C$ such that the rank of $H(C)$ does not exceed that of $H(B)$ and $\varphi: H(C) \rightarrow \HM$ is unitary. Theorems \ref{invsinggtn} and \ref{invrational} in this paper provide some additional information of the structure of the $M_z$-invariant subspaces of $H(b)$ when $b$ is a rational function.

\begin{acknowledgements}
S. Luo was supported by NNSFC (12271149, 11701167).
The authors thank Professor Stefan Richter for many helpful conversions. The authors thank the referee for careful reading of the paper and many crucial comments. The argument of showing $\mathcal{R}$ closed in Case 3 of Theorem \ref{invsubsftni} is due to the referee, our original argument is a little complicated.
\end{acknowledgements}


\begin{thebibliography}{99}

\bibitem {Ag90}
\textsc{J. Agler}, A disconjugacy theorem for Toeplitz operators. \textit{Amer. J. Math.} \textbf{112}, no. 1 (1990), 1-14.

\bibitem {AS956}
\textsc{J. Agler, M. Stankus}, $m$-isometric transformations of Hilbert
space. Part I, \textit{Integral Equations Operator Theory} \textbf{21}, no. 4 (1995), 383-429, Part II, \textit{Integral Equations Operator Theory} \textbf{23}, no. 1 (1995), 1-48,
Part III, \textit{Integral Equations Operator Theory} \textbf{24}, no. 4 (1996), 379-421.

\bibitem{Al93}
\textsc{A. Aleman}, \textit{The Multiplication Operator on Hilbert Spaces of Analytic Functions}. Habilitation, Fernuniversitaet hagen, 1993.

\bibitem{AlemanMalman}
\textsc{A. Aleman, B. Malman}, Hilbert spaces of analytic functions with a contractive backward shift. \textit{J. Funct. Anal.} \textbf{277}, no. 1 (2019), 157-199.

\bibitem{ARS96}
\textsc{A. Aleman, S. Richter, C. Sundberg}, Beurling's theorem for the Bergman space. \textit{Acta Math.} \textbf{177} (1996),  275-310.

\bibitem {Be48}
\textsc{A. Beurling}, On two problems concerning linear transformations in Hilbert space. \textit{Acta Math.} \textbf{81} (1948), 17 pp.

\bibitem {CR13}
\textsc{C. Costara, T. Ransford}, Which de Branges-Rovnyak spaces are
Dirichlet spaces (and vice versa)? \textit{J. Funct. Anal.} \textbf{265} (12) (2013), 3204-3218.

\bibitem {EKMR14}
\textsc{O. El-Fallah, K. Kellay, J. Mashreghi, T. Ransford}, \textit{A primer
on the Dirichlet space. Cambridge Tracts in Mathematics}, 203. Cambridge
University Press, Cambridge, 2014.

\bibitem {dBR66}
\textsc{L. de Branges, J. Rovnyak}, \textit{Square Summable Power Series}. Holt,
Rinehart and Winston, New York (1966).

\bibitem {FHR16}
\textsc{E. Fricain, A. Hartmann, W. T. Ross}, Concrete examples of
$\mathcal{H}(b)$ spaces. \textit{Comput. Methods Funct. Theory} \textbf{16}, no. 2 (2016), 287-306.

\bibitem {FHR19}
\textsc{E. Fricain, A. Hartmann, W. T. Ross}, Multipliers between range spaces of co-analytic Toeplitz operators, \textit{Acta Sci. Math. (Szeged)} \textbf{85}, no. 1-2 (2019), 215-230.

\bibitem {FM08}
\textsc{E. Fricain, J. Mashreghi}, Integral representation of the $n$th
derivative in de Branges-Rovnyak spaces and the norm convergence of its
reproducing kernel. \textit{Ann. Inst. Fourier (Grenoble)} \textbf{58} (6), (2008) 2113-2135.

\bibitem {FM16}
\textsc{E. Fricain, J. Mashreghi}, \textit{The theory of $\mathcal{H}(b)$
spaces}. Vols. 1 and 2, New Mathematical Monographs, 20 and 21, Cambridge University Press,
Cambridge, 2016.

\bibitem {FMS15}
\textsc{E. Fricain, J. Mashreghi, D. Seco}, Cyclicity in non-extreme de Branges-Rovnyak spaces.
Invariant subspaces of the shift operator, 131-136, Contemp. Math., 638, Centre Rech. Math. Proc., Amer. Math. Soc.,
Providence, RI, 2015.

\bibitem {Gu14}
\textsc{C. Gu}, Elementary operators which are $m$-isometries, \textit{Linear
Algebra Appl.} \textbf{451} (2014), 49-64.

\bibitem {Gu15}
\textsc{C. Gu}, The $(m,q)$-isometric weighted shifts on $l_{p}$ spaces,
\textit{Integral Equations Operator Theory} \textbf{82} (2015), 157-187.

\bibitem {Gu15b}
\textsc{C. Gu}, Functional calculus for $m$-isometries and related
operators on Hilbert spaces and Banach spaces, \textit{Acta Sci. Math. (Szged)} \textbf{81},
605-641 (2015).

\bibitem {GuLuo18}
\textsc{C. Gu and S. Luo}, Composition and multiplication operators on the derivative Hardy space $S^2(\mathbb{D})$. \textit{Complex Var. Elliptic Equ.} \textbf{63}, no. 5 (2018), 599-624.

\bibitem {KZ15}
\textsc{K. Kellay, M. Zarrabi}, Two-isometries and de Branges-Rovnyak
spaces. \textit{Complex Anal. Oper. Theory} \textbf{9}, no. 6 (2015), 1325-1335.

\bibitem {LGR}
\textsc{S. Luo, C. Gu, S. Richter}, Higher order local Dirichlet integrals and de Branges-Rovnyak spaces, \textit{Adv. Math.} \textbf{385} (2021), Paper No. 107748, 47 pp.

\bibitem {Ri88}
\textsc{S. Richter}, Invariant subspaces of the Dirichlet shift, \textit{J.
Reine Angew. Math.} \textbf{386} (1988), 205-220.

\bibitem {Ri91}
\textsc{S. Richter}, A representation theorem for cyclic analytic two-isometries, \textit{Trans. Amer. Math. Soc.} \textbf{328} (1991),
325-349.

\bibitem {Ri}
\textsc{S. Richter}, \textit{Private communication}, 2019.

\bibitem {RS91}
\textsc{S. Richter, C. Sundberg}, A formula for the local Dirichlet
integral, \textit{Michigan Math. J.} \textbf{38} (1991), 355-379.

\bibitem {RS92}
\textsc{S. Richter, C. Sundberg}, Multipliers and invariant subspaces in
the Dirichlet space, \textit{J. Operator Theory} \textbf{28}, no. 1 (1992), 167-186.

\bibitem {RN90}
\textsc{F. Riesz, B. Sz.-Nagy}, \textit{Functional analysis}. Dover Books on
Advanced Mathematics. Dover Publications, Inc., New York (1990).

\bibitem {Sa86}
\textsc{D. Sarason}, Doubly shift-invariant spaces in $H^{2}$. \textit{J.
Operator Theory} \textbf{16}, no. 1, 75-97 (1986).

\bibitem {Sa94}
\textsc{D. Sarason}, \textit{Sub-Hardy Hilbert Spaces in the Unit Disc}. Wiley, New York (1994).

\end{thebibliography}
\end{document}